\documentclass[11pt]{article}
\usepackage{amssymb}
\usepackage{amsmath}
\usepackage{amsfonts}
\usepackage{amscd}
\usepackage{amsthm}
\usepackage{color}
\usepackage{verbatim}
\usepackage{color}
\usepackage[all,cmtip]{xy}
\usepackage{enumerate}
\usepackage [latin1]{inputenc}
\usepackage{ulem}

\date{\today}

\def\pee#1{\hbox{$ {\mathbf P}^{#1}$}}
\def\OP#1{\hbox{${\mathcal O}_{{\mathbf P}^{#1}}$}}
\def\coker{\hbox{\text{coker\ }}}
\def\im{\hbox{\text{im\ }}}

\theoremstyle{plain}
\newtheorem{theorem}{Theorem}[section]
\newtheorem{proposition}[theorem]{Proposition}
\newtheorem{remark}[theorem]{Remark}
\newtheorem{lemma}[theorem]{Lemma}
\newtheorem{definition}[theorem]{Definition}
\newtheorem{corollary}[theorem]{Corollary}

\newcommand\sA{{\mathcal A}}
\newcommand\sB{{\mathcal B}}
\newcommand\sC{{\mathcal C}}
\newcommand\sD{{\mathcal D}}
\newcommand\sE{{\mathcal E}}
\newcommand\sF{{\mathcal F}}
\newcommand\sG{{\mathcal G}}

\newcommand\sL{{\mathcal L}}

\newcommand\sO{{\mathcal O}}
\newcommand\sP{{\mathcal P}}
\newcommand\sQ{{\mathcal Q}}

\date{}

\begin{document}
\title{Rank two bundles on \pee n with isolated cohomology}

\author{F. Malaspina and A.P. Rao
\vspace{6pt}\\
 {\small  Politecnico di Torino}\\
{\small\it  Corso Duca degli Abruzzi 24, 10129 Torino, Italy}\\
{\small\it e-mail: malaspina@calvino.polito.it}\\
\vspace{6pt}\\
{\small   University of Missouri-St.Louis}\\
 {\small\it St.Louis, MO 63121, United States}\\
{\small\it e-mail: raoa@umsl.edu}}
\maketitle \def\thefootnote{}

\bibliographystyle{plain}
\maketitle

\begin{abstract}The purpose of this paper is to study minimal monads associated to a rank two vector bundle $\sE$ on \pee n. In particular, we study situations where  $\sE$  has $H^i_*(\sE) =0$ for $1<i<n-1$, except for  one pair of values $(k,n-k)$. We show that on \pee 8, if $H^3_*(\sE)=H^4_*(\sE)=0$, then $\sE$ must be decomposable. More generally, we show that for $n\geq 4k$, there is no indecomposable bundle $\sE$ for which all intermediate cohomology modules except for $H^1_*, H^k_*, H^{n-k}_*, H^{n-1}_*$  are zero.
\end{abstract}

\section*{Introduction}

 It has been difficult to disprove the existence of an indecomposable rank two bundle $\sE$ on \pee n for large $n$. Most known results have been obtained by imposing other conditions on $\sE$ to show that $\sE$ cannot exist or must be split. For example, the so-called Babylonian  condition which requires $\sE$ to be extendable to \pee {n+m} for every $m$ has been studied by a number of people including Barth and van de Ven \cite {B-V} and  Coanda and Trautmann \cite{C-T}. Numerical criteria that force splitting are found again in Barth and van de Ven, where for a normalized rank two bundle with second Chern class $a$ and with splitting type $\sO_l(-b) \oplus \sO_l(b)$ on the general line $l$, a function $f(a,b)$ is found such that if $n > f(a,b)$, then a bundle on \pee n with these invariants must be split.

Cohomological criteria for forcing the splitting of $\sE$ start with Horrocks \cite{Horrocks}. If $S$ is the polynomial ring corresponding to \pee n, then $H^i_*(\sE)$ (defined as $\oplus_{\nu}H^i(\pee n, \sE(\nu))$)  is an $S$-module. The intermediate cohomology modules $H^i_*sE), 1\leq i \leq n-1$ are all graded modules of finite length and there is a strong relationship between $\sE$ and its intermediate cohomology modules. He shows that if $H^i_*(\sE)=0$ for all $i$ with $i\leq i \leq n-1$, then $\sE$ is split. Moreover Horrocks in \cite{Horrocks} established that a vector bundle on \pee n
is determined up to isomorphism and up to a sum of line bundles ({\it i.e.} up to stable equivalence) by its collection of
intermediate cohomology modules and also a certain collection of extension classes involving these modules. This correspondence has been generalized to any ACM varieties in \cite{MM}. The Syzygy Theorem (\cite{EG}, \cite{E}) shows that for a rank two bundle $\sE$, it is enough to know that $H^1_*(\sE) =0$ to force splitting. In \cite{M-P-R}, it is shown that for a indecomposable rank two bundle on \pee n, in addition to $H^1_*(\sE)$ and $H^{n-1}_*(\sE)$ being non-zero, some intermediate cohomology module $H^k_*(\sE)$ ($1<k<n-1$) (and hence also $H^{n-k}_*(\sE)$) must be non-zero. Various calculations in \cite{M-R} and \cite{M-P-R} show that there are limitations on the module structure of $H^1_*(\sE)$ and $H^2_*(\sE)$ for some values of $n$.

In this paper, we study situations where a rank two bundle $\sE$ on \pee n has $H^i_*(\sE) =0$ for $1<i<n-1$, except for  one pair of values $(k,n-k)$. We describe the minimal monads associated to $\sE$. We show that on \pee 8, if $H^3_*(\sE)=H^4_*(\sE)=0$, then $\sE$ must be decomposable. More generally, we show that for $n\geq 4k$, there is no indecomposable bundle $\sE$ for which all intermediate cohomology modules except for $H^1_*, H^k_*, H^{n-k}_*, H^{n-1}_*$  are zero. The proof utilizes the space between $k$ and $n-k$ when $n\geq 4k$ for making cohomological computations.

\section{Monads for rank two vector bundles on \pee{n}}

Let  $\sE$ be an indecomposable  rank two vector bundle on \pee{n}. If $S$ is the polynomial ring on $n+1$ variables, let $N_i= H^i_*(\sE) = \oplus_\nu H^i(\sE(\nu))$ be the finite length graded $S$-module over $S$, for $1\leq i \leq n-1$.
By the Syzygy Theorem, both $N_1$ and $N_{n-1}$ are non-zero modules.
Barth's construction of a minimal monad for $\sE$ gives a complex
 $$ 0 \to \sA \xrightarrow{\alpha} \sP \xrightarrow{\beta} \sB \to 0, $$
where $\sP$ is a bundle with $H^i_*(\sP)=0$ for $i=1$ and $i=n-1$, and  where $\sA, \sB$ are free bundles.  Let $\sG$ be kernel $\beta$.
We have two sequences  \begin{equation} \begin{split}0 \to \sG \to \sP \to \sB \to 0, \\
0\to  \sA \to \sG \to \sE \to 0,\end{split}\end{equation} from which we see that $H^i_*(\sE) = H^i_*(\sG) $ for $1 \leq i \leq n-2$, while $H^{n-1}_*(\sG) =0$, and  $H^i_*(\sE) = H^i_*(\sP) $ for $2 \leq i \leq n-2$,  while $H^{1}_*(\sP)= H^{n-1}_*(\sP)=0$.

They give rise to
\begin{equation} \label{wedgetwos}\begin{split} 0\to \wedge^2\sG \to \wedge^2 \sP \to \sB\otimes \sP \to S^2\sB \to 0, \\
0 \to S^2\sA \to \sA \otimes \sG \to \wedge^2 \sG \to \wedge^2 \sE\to 0.
\end{split}\end{equation}

\begin{lemma}\label{H1(wedge2) is non-zero}  If $H^2_*(\sE)=0$, then $H^1_*(\wedge^2 \sP)$ and $H^{n-1}_*(\wedge^2 \sP)$ are non-zero. If $H^l_*(\sE) = 0$ for some $l$, with $2 \leq l \leq n-2$, then $H^l_*(\wedge^2 \sP) =0.$
\end{lemma}
\begin{proof}
See \cite{M} Theorem 2.2. for the first part. Next, suppose $H^l_*(\sE) = 0$ for some $l$, with $2 \leq l \leq n-2$. So $N_l = N_{n-l}=0$ by Serre duality.

Since $\sG$ and $\sP$  have $H^l_* =0$ as well, it follows from equation (\ref{wedgetwos}), that   $H^l_*(\wedge^2 \sG) = 0$ and hence $H^l_*(\wedge^2 \sP) =0.$
\end{proof}

\begin{lemma} \label{cohomology of wedge 2 P} Let $2\leq t \leq n-2$. Let $A =H^0_*(\sA), B = H^0_*(\sB)$. There is an exact sequence
$$ A\otimes N_t \to H^t_*(\wedge^2 \sP) \to B\otimes N_t $$
which is injective on the left if $t\geq 3$ and $N_{t-1}=0$, and is surjective on the right if $t\leq n-3$ and $N_{t+1}=0$.
\end{lemma}
\begin{proof} Break up the first sequence in \ref{wedgetwos} as $0\to \wedge^2\sG \to \wedge^2 \sP \to  \sD \to 0$, $0 \to \sD \to \sB\otimes \sP \to S^2\sB \to 0$. We get long exact sequences
$$ H^{t-1}_*(\sD) \to H^{t}_*(\wedge^2\sG)  \to H^{t}_*(\wedge^2 \sP) \to H^{t}_*(\sD) \to H^{t+1}_*(\wedge^2\sG),$$
where $H^t_*(\sD) \cong B\otimes N_t$ (always) and $H^{t-1}_*(\sD)  \cong B\otimes N_{t-1}$ provided $t \geq 3$.
Likewise break up the second sequence as $0 \to S^2\sA \to \sA \otimes \sG \to \sC \to 0,$ $0\to \sC \to \wedge^2 \sG \to \wedge^2 \sE\to 0.$ We see that $H^i_*(\wedge^2\sG) \cong H^i_*(\sC)$ for $i=t,t+1$, $H^t_*(\sC)\cong A \otimes N_t $ and when $t \leq n-3$, $H^{t+1}_*(\sC)\cong A \otimes N_{t+1} $.

\end{proof}

The following proposition is a typical one that shows that a minimal monad for a rank two bundle is built very minimally out of the cohomological data for $\sE$. Other examples of such a result can be found in \cite{Rao}, \cite{M-R}. Decker (\cite{Decker})  has conjectured such a minimality for rank two bundles on $\pee 4$.

\begin{proposition} Suppose $\sE$ is a non-split rank two bundle on $\pee n$ ($n\geq 6$), with $H^l_*(\sE)=0$ for some $l$ with $2 \leq l \leq n-2$. Then in the minimal monad for $\sE$, the bundle $\sP$ has no line bundle summands.
\end{proposition}

\begin{proof}
Note that the statement is vacuous for  $n= 4, 5$, since $\sE$ will be split by \cite{M-P-R}.  So assume that $n\geq 6$ and that $\sE$ satisfies  $H^l_*(\sE) =0$ for some $2\leq l \leq n-2$. By  \cite{M-P-R}, there must also be a $j$ such that $H^j_*(\sE) \neq 0$ for some $2 \leq j \leq n-2$.

We may choose $l$  to be the lowest value with $H^l_*(\sE) =0$ and let us suppose that $l\geq 3$. Then $H^{l-1}_*(\sE) = N_{l-1}\neq 0$. Consider the exact sequence using Lemma \ref{cohomology of wedge 2 P} (with $t= l-1$)
$$ A\otimes N_{l-1} \to H^{l-1}_*(\wedge^2 \sP) \to B\otimes N_{l-1}\to 0.$$

Now if $\sP \cong \sQ \oplus \OP {} (a)$, then $H^{l-1}_*(\wedge^2 \sP) \cong H^{l-1}_*(\wedge^2 \sQ) \oplus [S(a) \otimes N_{l-1}]$, where $N_{l-1} \neq 0$.
The map $S(a)\otimes N_{l-1} \to B\otimes N_{l-1}$ in the sequence is induced by the map $\OP {}(a) \otimes \sP_k \xrightarrow{\beta_2\otimes I} \sB \otimes \sP_k$, where $\beta = [\beta_1, \beta_2]$ in the monad for $\sE$.

The map $A\otimes N_{l-1} \to S(a)\otimes N_{l-1}$ is induced by the map $\sA\otimes \sG \to \wedge^2 \sG \hookrightarrow \wedge^2 \sP \twoheadrightarrow \OP{}(a) \otimes \sP$, hence by $\sA \otimes \sP \xrightarrow{\alpha_2 \otimes I} \sL\otimes \sP$ if $\alpha = [\alpha_1, \alpha_2]^T$ in the monad.

The sequence above now reads
$$   A\otimes N_{l-1} \xrightarrow {\bmatrix *\\ \alpha_2\otimes I\endbmatrix} H^{l-1}_*(\wedge^2 \sQ) \oplus [S(a) \otimes N_{l-1}] \xrightarrow {\bmatrix *, \beta_2\otimes I\endbmatrix} B \otimes N_{l-1} \to 0 $$

If we tensor the sequence  by the quotient $k = S/(X_0,\dots, X_{n+1})$, since  the matrix $\beta_2$ is a minimal matrix, $(\beta_2\otimes I)\otimes k =0$, hence $[S(a)\otimes N_{l-1}\otimes k]$ is inside the kernel of ${\bmatrix *, \beta_2\otimes I \endbmatrix} \otimes k$.  By exactness, $S(a)\otimes N_{l-1} \otimes k$ is inside the image of $(\alpha_2\otimes I)\otimes k$.  which is not possible since $\alpha_2$ is also a minimal matrix.

It remains to study the case where $l=2$. There is a value  $l'$ between $3$ and $n-3$ for which $H^{l'}_*(\sE) = N_{l'} \neq 0$ and $H^{l'+1}_*(\sE) = 0$. We now have an exact sequence of non-zero $S$-modules
$$ A\otimes N_{l'} \to H^{l'}_*(\wedge^2 \sP) \to B\otimes N_{l'}\to 0$$
and we repeat the earlier argument to get a contradiction.
\end{proof}

\begin{definition} A rank two bundle $\sE$ on \pee n, $n\geq 6$ will be said to have isolated cohomology of type $(n,k)$ if there exists an integer $k$, $1<k\leq\frac n2$, with $H^k_*(\sE)$ and $H^{n-k}_*(\sE)$ non-zero modules, and $H^i_*(\sE) =0$ for $i \neq 1, k, n-k, n-1$.
\end{definition}
\begin{remark} By Lemma\ref{H1(wedge2) is non-zero} we get that if $\sE$ has isolated cohomology of type $(n,k)$, then $H^i_*(\wedge^2 \sP) =0$ for $i \neq 1, k, n-k, n-1$.
\end{remark}

A special case in the definition is when the middle cohomology is not zero, {\it ie.} of type $(n, k)$, where $n$ is even, equal to $2k$, and the only non-zero cohomology modules are $H^1_*(\sE), H^k_*(\sE),H^{n-1}_*(\sE)$.

Note that the conditions that $H^1_*(\sE), H^{n-1}_*(\sE)$ are both non-zero for an indecomposable rank two bundle follows from the Syzygy Theorem. In \cite{M-P-R}, it is proved that for an indecomposable rank two bundle on \pee n, $n\geq 4$, at least one cohomology module $H^l_*(\sE)$ must be non-zero with $1<l<n-1$. The reason $n$ is chosen to be $\geq 6$ in the definition is that first, the definition is vacuous for $n=2,3$ and second, for $n=4,5$, $k$ must be $2$, and the definition made is  always satisfied by any possible indecomposable rank two bundle on \pee 4 or \pee 5, hence imposes no restrictions.

Let  $\sP_k(N)$ be the $k^{th}$ syzygy bundle of the finite length module $N$. By this, we mean that in a minimal free resolution for $N$ over the polynomial ring $S$:
$$0\to L_{n+1}\xrightarrow{f_{n+1}} L_{n} \to \dots \to L_{k+1}\xrightarrow{f_{k+1}} L_k\to \dots \to L_1\xrightarrow{f_1} L_0 \to N\to 0,$$
 $P_k(N)$ will denote the image of $f_{k+1}$ and $\sP_k(N)$ will denote the sheafification of $P_k(N)$. Hence $H^k_*(\sP_k(N))=N$, with $H^i_*(\sP_k(N))= 0 {\text {\ when\  }}i \neq 0, k, n$.  According to \cite{Horrocks}, if $\sP$ is any bundle on \pee n with the property that $H^k_*(\sP) =N$ and $H^i_*(\sP)= 0 {\text {\ when\  }}i \neq 0, k, n$, then
 $\sP \cong \sP_k(N) \oplus \sF$ where $\sF$ is a direct sum of line bundles.

\begin{lemma} Let $\sP$ be a vector bundle on \pee n with non-zero cohomology modules $H^k_*(\sP) = N$, $H^l_*(\sP) = M$ for $1 \leq k < l \leq n-1$, and with $H^i_*(\sP)= 0 {\text {\ when\  }}i \neq 0, k, l, n$. Then there is an exact sequence
$$ 0 \to \sP_k(N) \to \sP \oplus \sF  \to \sP_l(M) \to 0,$$
where $\sF$ is some free bundle.
\end{lemma}
\begin{proof} This too follows from \cite{Horrocks}. Letting $P$ denote $H^0_*(\sP)$, form an exact sequence (by partially resolving $P^\vee$)
$$0 \to P \to L_k \to L_{k-1}\to \dots \to L_1 \to A \to N \to 0,$$
where $A$ is not a free module. Compare this with a truncated minimal free resolution of $N$:
$$ 0 \to P_k(N) \to L_k' \to L_{k-1}'\to \dots \to L_1' \to L_0' \to N \to 0.$$
The induced map $P_k(N) \to P$ gives a map $\sP_k(N) \to \sP$ which is an isomorphism at the cohomology level $H^k_*$.
Minimally add a free module $F$ to $P$ to force a surjection $P^\vee \oplus F^\vee \to P_k(N)^\vee$. This gives an inclusion of bundles
$ \sP_k(N) \to \sP \oplus \sF$ whose cokernel is $\sP_l(M)\oplus\sF'$ where $\sF'$ is a free bundle (since it has only $H^l_*$ intermediate cohomology).
We notice that both for $k=1$ and for $k>1$, the map $H^1_*(\sP_k(N)) \to H^1_*(\sP \oplus \sF)$ is an isomorphism, so we get a surjection from $H^0_*(\sP \oplus \sF)$ to $H^0_*(\sP_l(M)\oplus\sF')$. By the minimality of $F$ we may conclude that $\sF'=0$
\end{proof}

Summarizing this below, we get:

\begin{proposition}\label{type 1} Let $\sE$ be a rank two bundle on \pee n, $ n\geq 6$ with isolated cohomology of type $(n,k)$ with $H^k_*(\sE)=N$, for some $k$ strictly between $1$ and $\frac n2$.  Then
 $\sE$ has the monad
$$ 0 \to \sA \xrightarrow{\alpha} \sP \xrightarrow{\beta} \sB \to 0, $$
where \begin{itemize}
\item $\sP$ satisfies an exact sequence $ 0 \to \sP_k(N) \to \sP \oplus \sF  \to \sP_{n-k}(M) \to 0,$ where $\sF$ is some free bundle, $M = H^{n-k}_*(\sE)$ (which can be identified with $N^\vee$ up to twist).
\item $H^i_*(\wedge^2 \sP) =0$ for $i \neq 1, k, n-k, n-1$.
\item $H^1_*(\wedge^2 \sP)$ and $ H^{n-1}_*(\wedge^2\sP)$ are non-zero if $k\neq 2$
\end{itemize}
\end{proposition}

\vskip 5pt
\vskip 10pt

In the case left out in the above proposition, where $\sE$ has isolated middle cohomology with $n=2k$ and with $H^k_*(\sE)=N\neq 0$ equal to the only  non-zero cohomology module in the range $1<i<n-1$,  the monad for $\sE$ has the form
$$ 0\to \sA \to \sP_k(N) \to \sB \to 0.$$
Also there is a short exact sequence
$$0 \to A\otimes N \to H^k_*(\wedge^2 \sP_k(N)) \to B \otimes N \to 0.$$

Thus
\begin{proposition}\label{type 2} Let $\sE$ be a rank two bundle on \pee n, $n=2k, n\geq 6$, with $H^k_*(\sE)=N$, $H^i_*(\sE)=0, i \neq 1,k,n$. Let $\sP_k$ be the $k^{th}$ syzygy bundle of $N$ where $\sP_k$ is the sheafification of $P_k$ with $P_k = \text {Image of\ } (f_{k+1}: L_{k+1} \to L_k)$ in a minimal free resolution of $N$. Then
 $\sE$ has the monad
$$ 0 \to \sA \xrightarrow{\alpha} \sP_k \xrightarrow{\beta} \sB \to 0, $$
where $\sA, \sB$ are sheafifications of free summands
$A,B$ of $L_{k+1}$ and $L_k$ respectively and where $\alpha, \beta$ are induced by $f_{k+1}$. Furthermore
\begin{itemize}
\item $H^i_*(\wedge^2 \sP_k) =0$ for $i \neq 1, k, n-1$.
\item The induced sequence $0 \to A\otimes N \to H^k_*(\wedge^2 \sP_k) \to B \otimes N \to 0$ is exact
\item $H^1_*(\wedge^2 \sP_k)$ and $ H^{n-1}_*(\wedge^2\sP_k)$ are non-zero.
\end{itemize}
\end{proposition}

\begin{proof} The only item to verify is that $\sA, \sB$ are sheafifications of free summands
$A,B$ of $L_{k+1}$ and $L_k$ respectively and that $\alpha, \beta$ are induced by $f_{k+1}$.  Since $L_{k+1} \to P_k$ is surjective, $\alpha: A \to P_k$ factors through $\tilde\alpha : A \to L_{k+1}$. Likewise, since $L_k^\vee \to P_k^\vee$ is surjective, $\beta^\vee: B^\vee \to P_k^\vee$ factors through $\tilde \beta^\vee: B^\vee \to L_k^\vee$. It remains to show that the matrices $\tilde\alpha, \tilde \beta$ have full rank when tensored by $k$.

The map $H^k_*(\wedge^2 \sP_k) \to B \otimes N \to 0$ in the short sequence above  is obtained from $\wedge^2 \sP_k \to \sB \otimes \sP_k$ where  $p\wedge q$ maps to $\beta(p)\otimes q - \beta(q)\otimes p$.  This factors through $\sL_k \otimes \sP_k$ via the lift $\tilde \beta$. In particular, the map $L_k\otimes N \to B\otimes N$, given by $\tilde\beta \otimes I$, is onto. Hence so is $(\tilde\beta \otimes k ) \otimes I$, a map of vector spaces. Hence the matrix $\tilde\beta \otimes k$ has rank equal to the rank of $B$. So $B$ is a direct summand of $L_k$.

The map $0 \to A\otimes N \to H^k_*(\wedge^2 \sP_k) $ is obtained from $H^k_* (\sA \otimes \sG ) \cong H^k_*(\wedge^2\sG) \hookrightarrow H^k_*(\wedge^2\sP_k)$, which in turn is obtained from $\sA\otimes \sG \to \wedge^2\sG \hookrightarrow \wedge^2\sP_k$, where $a\otimes g$ maps to $\alpha (a)\wedge g$ in $\wedge^2\sP_k$. This map $\sA\otimes \sG \to \wedge^2\sP_k$ factors through $\sL_{k+1}\otimes \sG$, vial the lift $\tilde \alpha$.

It follows that the injection $A\otimes N \to H^k_*(\wedge^2\sP_k)$ factors through $A\otimes N \to L_{k+1}\otimes N$, by the map $\tilde\alpha \otimes I$. This must also be injective. Choose a socle element $n$ in $N$ (an element that is annihilated by all linear forms in $S$). The submodule generated by $n$, $\langle n \rangle$, is a one-dimensional vector space and $A\otimes \langle n\rangle$ is mapped injectively by $\tilde\alpha \otimes I$ to $L_{k+1}\otimes N$. Since the image of $\tilde\alpha \otimes I$ on $A\otimes \langle n\rangle$ is the same as the image of $(\tilde\alpha \otimes k) \otimes I$ on $(A\otimes k)\otimes \langle n\rangle$, it follows that the rank of the matrix  $\tilde\alpha \otimes k$ has rank equal to the rank of $A$. Thus $A$ is a direct summand of $L_{k+1}$.

\end{proof}

We now review a result of Jyotilingam \cite{J} about cohomology modules of tensor products, applying it to the special case of syzygy bundles for our purposes. In the theorem below, $N$ and $M$ will be graded finite length $S$-modules where $S = k[X_0, X_1, \dots, X_n]$ corresponding to $\pee n$. $\sP_k(N)$ and $\sQ_l(M)$ will indicate syzygy bundles obtained from minimal free resolutions of $N$ and $M$.  Note that in the minimal free resolution
$$ 0 \to L_{n+1} \to L_n \to \dots \to L_1 \to L_0 \to N \to 0, $$
when we tensor by $M$, the map $L_{n+1}\otimes M \to L_n \otimes M$, cannot be injective since $M$ has finite length, hence $Tor_{n+1}^S(N,M) \neq 0$, and by Lichtenbaum's theorem \cite{L}  $Tor_{i}^S(N,M) \neq 0$ for all $i \leq n+1$.

\begin{theorem}\label{Jyoti}  Let  $N$ be a finite $S$-module and let $\sP_k$ be its $k^{th}$ syzygy bundle on $\pee n$, with $k\geq 1$. Let $\sQ$ be a bundle on $\pee n$ with $H^l_*(\sQ)=M \neq 0$, with $k \leq l \leq n-2$, and with $H^i_*(\sQ) =0$ for $i=l-1, l-2, \dots, l-k+2$. Then $H^{l+1}_*(\sP_k \otimes \sQ) \neq 0$.
\end{theorem}
\begin{proof} The cases $k=1$ and $k=2$ require no conditions on $H^{l-1}_*(\sQ)$. When $k=1$, we get the sequence $H^l_*(\sL_1\otimes \sQ) \to H^l_*(\sL_0\otimes \sQ) \to H^{l+1}_*(\sP_1\otimes \sQ) \to 0$
and the map $L_1\otimes M \to L_0\otimes M$ can never be surjective. When $k>1$, consider the diagram obtained from
the sequences $0 \to \sP_i \otimes \sQ \to \sL_i\otimes \sQ \to \sP_{i-1} \otimes \sQ\to 0, \  i=k, k-1, k-2$ (with $\sP_j=0$ if $j<0$ and $\sP_0=\sL_0$):
$$ \begin{CD} &L_{k} \otimes M &= 	& L_{k} \otimes M & &H^{l-1}_*(\sP_{k-3}\otimes \sQ) &\\
				&	\downarrow       &       &     \downarrow{\gamma}  &      &       \downarrow \\
				&H^{l}_*(\sP_{k-1} \otimes \sQ) &\xrightarrow{\alpha} 	&L_{k-1} \otimes M &\xrightarrow{\beta} &H^{l}_*(\sP_{k-2}\otimes \sQ) &\\
				&	\downarrow{\mu}       &       &     \downarrow{\delta}   &      &       \downarrow \\
                        	&H^{l+1}_*(\sP_k \otimes \sQ)       &    & L_{k-2} \otimes M   &=  &L_{k-2} \otimes M & \\
                          \end{CD}$$
The vanishing conditions on $H^i_*(\sQ)$ show that $H^{l-1}_*(\sP_{k-3}\otimes \sQ) = H^{l-2}_*(\sP_{k-4}\otimes \sQ) = \dots = H^{l-k+2}_*(\sL_{0}\otimes \sQ) =0$. So $\ker \delta = \im \alpha$ and the diagram induces a surjection $\im \mu \to Tor_{k-1}(N,M)$. By Lichtenbaum's theorem,  $H^{l+1}_*(\sP_k \otimes \sQ) \neq 0$.
\end{proof}
 \vskip .5cm

 \section{Isolated cohomology of type $(n,k)$, with $n\geq 4k$}

In this section, we will prove that there are no indecomposable rank two bundles on \pee n with isolated cohomology of type $(n,k)$, where $n \geq 4k$. We study the sequence $ 0 \to \sP_k(N) \to \sP \oplus \sF  \to \sP_{n-k}(M) \to 0$ of Proposition \ref{type 1}. We will need to pay special attention to the case where $N$ is a cyclic module. Hence the following lemma.

\begin{lemma} \label{cyclic modules}  Let $N$ be a graded cyclic S-module. For the corresponding syzygy bundle $\sP_2(N)$ on \pee n,  $H^3_*(S^2\sP_2(N)) =0$ and $H^3_*(\wedge^2\sP_2(N)) \neq 0$.
\end{lemma}
\begin{proof} From the sequence  $0 \to \sP_2 \to \sL_2 \to \sP_1 \to 0$ obtained from a minimal resolution of $N$, it suffices to show that  the map
$H^1_*(\sL_2 \otimes \sP_1) \to H^1_*(\wedge^2 \sP_1)$
 is surjective to prove that $H^3_*(S^2\sP_2(N)) =0$. This map can be studied using the natural commuting diagram
				
$$ \begin{CD} 0 \to 	&\sL_2 \otimes \sP_1 &\to 	&\sL_2 \otimes \sL_1 &\to &\sL_2\otimes \sL_0 &\to 0\\
				&	\downarrow       &       &     \downarrow   &      &       \downarrow \\
                        0 \to 	&\wedge^2 \sP_1        &\to    & \wedge^2 \sL_1   &\to  &\sL_1 \otimes \sL_0 &\to S^2 \sL_0 \\
                          \end{CD}$$

It simplifies when $\sL_0$ has rank one, where without loss of generality, we can take $\sL_0$ to be \OP n , yielding

 $$ \begin{CD} 0 \to 	&\sL_2 \otimes \sP_1 &\to 	&\sL_2 \otimes \sL_1 &\to &\sL_2 & \to 0\\
				&	\downarrow       &       &     \downarrow   &      &       \downarrow \\
                        0 \to 	&\wedge^2 \sP_1        &\to    & \wedge^2 \sL_1   &\to  &\sP_1 &\to 0 \\
                          \end{CD}.$$

  Since $\sL_2$ surjects onto the global sections of $\sP_1$, it follows from the diagram of long exact sequences  of cohomology modules  that $H^1_*(\sL_2 \otimes \sP_1) \to H^1_*(\wedge^2 \sP_1)$ is onto.

  For the second part, we will show that $H^3_*(\sP_2\otimes \sP_2) \neq 0$. (this argument will be repeated later in a slightly different setting.) With $H^3_*(S^2\sP_2) =0$, since $H^3_*(\sP_2\otimes \sP_2)  = H^3_*(S^2\sP_2) \oplus H^3_*(\wedge^2\sP_2)$, the conclusion of the lemma  follows.

  Consider  $ 0 \to \sP_2\otimes \sP_2 \to \sL_2 \otimes \sP_2 \to \sL_1 \otimes \sP_2 \to \sL_0 \otimes \sP_2 \to 0.$
  From $0 \to \sP_1 \otimes \sP_2 \to \sL_1 \otimes \sP_2 \to \sL_0 \otimes \sP_2 \to 0$, we get
  $$H^2_*(\sP_1\otimes \sP_2) = \ker (L_1\otimes N \to L_0 \otimes N) = L_1\otimes N$$
  since $N$ is cyclic. Hence we get
  $$H^3_*(\sP_2\otimes \sP_2) = \coker (L_2\otimes N \to L_1 \otimes N)$$
  which is clearly non-zero.

\end{proof}

\begin{proposition}{\label{Split case}}Suppose $\sE$ on \pee n is a rank two bundle of type $(n,k)$ with $n\geq 7$, $k$ strictly less than $\frac n2$. Then the sequence $ 0 \to \sP_k(N) \to \sP \oplus \sF  \to \sP_l(M) \to 0,$  in Proposition \ref{type 1} is not-split.
\end{proposition}
\begin{proof}
Suppose  $\sP \oplus \sF = \sP_k(N) \oplus \sP_{n-k}(M)$. Neither $\sP_k(N)$ nor  $\sP_{n-k}(M)$ has any line bundle summands, hence $\sP = \sP_k(N) \oplus \sP_{n-k}(M)$. So $\wedge^2 \sP$ has summands $\sP_k(N) \otimes \sP_{n-k}(M)$ and $\wedge^2\sP_k(N)$.
If $k>2$ then using Proposition \ref{Jyoti}, $H^{n-k+1}_*(\sP_k(N) \otimes \sP_{n-k}(M))$ is non-zero which contradicts the requirement in Proposition \ref{type 1} that $H^{n-k+1}_*(\wedge^2 \sP) =0$.

 If $k=2$, there are two cases: if $N$ is cyclic, then $H^3_*(\wedge^2\sP_2(N)) \neq 0$ by Lemma \ref{cyclic modules}, which contradicts Proposition \ref{type 1} since $n-k > 3$ when $n\geq 6$.

 If $N$ is non-cyclic, then from the sequences $0 \to \sP_2(N) \to \sL_2 \to \sP_1(N)\to 0$ and $0\to \sP_1 \to \sL_1 \to \sL_0 \to 0$, we get $H^4_*(\wedge^2\sP_2(N)) \neq 0$. This a contradiction to Proposition \ref{type 1} when $n\geq 7$.
\end{proof}

\begin{remark} The case $n=6, k=2$ is not answered above. A weaker argument can be made here that even though $\sP = \sP_k(N) \oplus \sP_{n-k}(M)$, $N$ itself is neither cyclic nor a direct sum of submodules $N_1\oplus N_2$.
\end{remark}

\begin{theorem}\label{P7} Let $\sE$ be a rank two vector bundle on $\pee 8$ with $H^3_*(\sE)=H^4_*(\sE)=0$, then $\sE$ splits.
  \end{theorem}
  \begin{proof} Let $N=H^2_*(\sE)$ and $M=H^6_*(\sE)$. Both are non-zero unless $\sE$ splits. By Proposition \ref{Split case} (with $k=2$), we know that the sequence \label{Horr-type1b} below is non-split.
\begin{equation} \label{Horr-type1b} 0 \to \sP_2(N) \to \sP \oplus \sF  \to \sP_{6}(M) \to 0. \end{equation}
The proof will analyze the consequences of the  two sequences below obtained from sequence  \ref{Horr-type1b}.
 \begin{equation} \ \label{first8} 0 \to S^2\sP_2(N) \to \sP_2(N) \otimes [\sP \oplus \sF ]  \to \wedge^2\sP \oplus [\sP\otimes \sF] \oplus \wedge^2\sF  \to \wedge^2\sP_{6}(M) \to 0, \end{equation}
\begin{equation} \ \label{second8} 0 \to \wedge^2\sP_2(N) \to \wedge^2\sP \oplus [\sP\otimes \sF] \oplus \wedge^2\sF  \to \sP_{6}(M) \otimes [\sP \oplus \sF ]   \to S^2\sP_{6}(M) \to 0, \end{equation}

  \noindent {\bf Case 1} If $N$ is cyclic we look at the sequence (\ref{first8}).

  It breaks into
  \begin{equation}\label{breakup7} \begin{split} 0 \to S^2\sP_2(N) \to \sP_2(N) \otimes [\sP \oplus \sF ]  \to \sD \to 0, \\
 0 \to \sD \to \wedge^2\sP \oplus [\sP\otimes \sF] \oplus \wedge^2\sF  \to \wedge^2\sP_{6}(M) \to 0
 \end{split}
 \end{equation}
 $H^3_*(\sP_2(N)\otimes [\sP \oplus \sF ]) \neq 0$ by the same argument as in the second part of the proof of Lemma \ref{cyclic modules}, and by the same lemma, $H^3_*(S^2\sP_2(N))=0$.  Hence $H^3_*(\sD) \neq 0$ from the first sequence in (\ref{breakup7}).

 In the second sequence in (\ref{breakup7}), $H^3_*(\sP)=0$. Hence so is
 $H^3_*(\wedge^2\sP)$.  Finally, $\sP_{6}(M)$ fits into a sequence with free bundles
 $$ 0 \to \sL'_{9} \to \sL'_{8} \to \sL'_{7} \to \sP_{6} \to 0.$$
 This yields two exact sequences
 \begin{equation}\label{P_{6}} \begin{split}
 0 \to S^2\sP_{7}  \to S^2\sL'_{7} \to  \sL'_{7}\otimes \sP_{6}  \to \wedge^2\sP_{6}\to 0\\
 0 \to \wedge^2\sL'_{9}  \to \wedge^2\sL'_{8} \to  \sL'_{8}\otimes \sP_{7}  \to S^2\sP_{7}\to 0\\
  \end{split}
 \end{equation}

  From these, we can chase down $H^2_*(\wedge^2\sP_{6})$ to be equal to zero since $H^2_*(\sP_{6})=0, H^4_*(\sP_{7})=0, H^6_*(\wedge^2\sL'_{9})=0$.
  Hence $H^3_*(\sD)$ is both zero and non-zero, a contradiction.

  \noindent {\bf Case 2} If $N$ is non cyclic, we look at the sequence (\ref{second8})
  $$0 \to \wedge^2\sP_2(N)  \to \wedge^2\sP \oplus [\sP\otimes \sF] \oplus \wedge^2\sF \to \sP_{6}(M) \otimes [\sP \oplus \sF ]   \to S^2\sP_{6}(M) \to 0.$$
  It breaks into
  \begin{equation}\label{breakup27} \begin{split} 0 \to \wedge^2\sP_2(N) \to \wedge^2\sP \oplus [\sP\otimes \sF] \oplus \wedge^2\sF  \to \sD \to 0, \\
 0 \to \sD \to \sP_{6}(M) \otimes [\sP \oplus \sF ]   \to S^2\sP_{6}(M) \to 0
 \end{split}
 \end{equation}

  From $$0 \to S^2 \sP_1(N) \to S^2 \sL_1 \to \sL_1 \otimes \sL_0 \to \wedge^2 \sL_0 \to 0,$$
  $$0 \to \wedge^2 \sP_2(N) \to \wedge^2 \sL_2 \to \sL_2 \otimes \sP_1 \to S^2 \sP_1(N) \to 0.$$ we get  $H^2_*(S^2\sP_1(N)) \neq 0$ and $H^4_*(\wedge^2\sP_2(N)) \neq 0$. Since $H^{4}_*(\sP)$ and $H^{4}_*(\wedge^2\sP)$ are zero, we obtain $H^{3}_*(\sD) \neq 0$.

 Again, in the second sequence in (\ref{breakup27}),  $H^{3}_*(\sP_{6}(M)\otimes \sF)=0$  and   $H^{3}_*(\sP_{6}(M)\otimes \sP)$ can be studied using a resolution for $\sP_{6}(M)$ and tensoring with $\sP$.
  $$ 0 \to \sL'_{9}\otimes \sP \to \sL'_{8}\otimes \sP \to  \sL'_{7}\otimes \sP \to \sP_{6}(M)\otimes \sP \to 0.$$

  Then $H^{3}_*(\sP_{6}(M)\otimes \sP) =0$ since $H^{3}_*(\sP), H^{4}_*(\sP) , H^{5}_*(\sP)$ are all zero.

  We compute $H^{2}_*(S^2\sP_{6}(M))$,  breaking up the resolution of $\sP_{6}$ (suppressing the letter $M$) into short exact sequences:
  \begin{equation}\label{P_{5}} \begin{split}
 0 \to \wedge^2\sP_{7}  \to \wedge^2\sL'_{7} \to  \sL'_{7}\otimes \sP_{6}  \to S^2\sP_{6}\to 0\\
 0 \to S^2\sL'_{9}  \to S^2\sL'_{8} \to  \sL'_{8}\otimes \sP_{7}  \to \wedge^2\sP_{7}\to 0
  \end{split}
 \end{equation}
 $H^{2}_*(S^2\sP_{6}(M))$ will vanish since $H^{2}_*(\sP_{6})$, $H^{4}_*(\sP_{7})$  and $H^{6}_*(S^2\sL'_{9})$ are all zero.

  \end{proof}

   \begin{corollary} Let $n\geq 8$. Let $\sE$ be a rank two vector bundle on $\pee n$ with $H^i_*(\sE)=0$ for $i=3,\dots n-3$. Then $\sE$ splits.
  \end{corollary}
  \begin{proof} Use induction on $n$. The case $n=8$ is proved in the above theorem. Assume the result for $n-1$. Let $\sE$ be a rank two vector bundle on $\pee n$ with $H^i_*(\sE)=0$ for $i=3,\dots n-3$. For a hyperplane $H$, by the restriction sequence in cohomology $$ H^{i}_*(\sE)\to H^i_*(\sE_H)\to H^{i+1}_*(\sE(-1))$$ we get that $H^i_*(\sE_H)=0$ for $i=3,\dots n-4$ on $\pee {n-1}$. So $\sE_H$ splits and hence also $\sE$.
   \end{proof}

The theorem above can be generalized to arbitrary $k$ using similar calculations.

 \begin{theorem} Let $n\geq 4k$, with $k>1$. Then there cannot exist a rank two bundle $\sE$ on \pee n,  for which the only non-zero intermediate cohomology modules are $H^1_*(\sE)$, $H^k_*(\sE) =N$, $H^{n-k}_*(\sE) = M $, and $H^{n-1}_*(\sE)$.
  \end{theorem}
  \begin{proof} The case $k=2$ was done in the Corollary above. So we assume that $k>2$. The proof will analyze the consequences of the sequence

\begin{equation} \label{Horr-type1} 0 \to \sP_k(N) \to \sP \oplus \sF  \to \sP_{n-k}(M) \to 0 \end{equation}
which is non-split by Proposition \ref{Split case}. We get the collateral sequence:

\begin{equation} \ \label{second} 0 \to \wedge^2\sP_k(N) \to \wedge^2\sP \oplus [\sP\otimes \sF] \oplus \wedge^2\sF  \to \sP_{n-k}(M) \otimes [\sP \oplus \sF ]   \to S^2\sP_{n-k}(M) \to 0, \end{equation}

  We will prove it using several cases.
    \vskip 5 pt
  \noindent {\bf Case 1}  The case where $N$ is cyclic, $k$ is even and $>2$.

  We look at the sequence (\ref{second}) which
   breaks into
  \begin{equation}\label{breakup2} \begin{split} 0 \to \wedge^2\sP_k(N) \to \wedge^2\sP \oplus [\sP\otimes \sF] \oplus \wedge^2\sF  \to \sD \to 0, \\
 0 \to \sD \to \sP_{n-k}(M) \otimes [\sP \oplus \sF ]   \to S^2\sP_{n-k}(M) \to 0
 \end{split}
 \end{equation}
 $H^3_*(\wedge^2\sP_2(N)) \neq 0$ by Lemma \ref{cyclic modules}. This yields $H^{2k-1}_*(\wedge^2\sP_k(N))\neq 0$, since $n>2k-1$. On the other hand, $H^{2k-1}_*(\sP)$ and $H^{2k-1}_*(\wedge^2\sP)$ are zero, since $k < 2k-1 < n-k$ when $n \geq 4k$.  Hence $H^{2k-2}_*(\sD) \neq 0$ using the first short exact sequence in (\ref{breakup2}).

 In the second sequence in (\ref{breakup2}), $H^{2k-2}_*(\sP_{n-k}(M)\otimes \sF)=0$ since $2k-2 \neq n-k$.
 $H^{2k-2}_*(\sP_{n-k}(M)\otimes \sP)$ can be studied using a resolution for $\sP_{n-k}(M)$ and tensoring with $\sP$.
  $$ 0 \to \sL'_{n+1}\otimes \sP \to \sL'_{n}\otimes \sP \to \dots \sL'_{n-k+2}\otimes \sP \to \sL'_{n-k+1}\otimes \sP \to \sP_{n-k}(M)\otimes \sP \to 0.$$

  Then $H^{2k-2}_*(\sP_{n-k}(M)\otimes \sP) =0$ provided $H^{2k-2}_*(\sP), H^{2k-1}_*(\sP) , \dots, H^{3k-2}_*(\sP)$ are all zero. Since $n \geq 4k$, $n-k > 3k-2$ and since $k>2$, $k< 2k-2$. Hence these vanishings hold.

  We compute $H^{2k-3}_*(S^2\sP_{n-k}(M))$,  breaking up the resolution of $\sP_{n-k}$ (suppressing the letter $M$) into short exact sequences:
  \begin{equation}\label{P_{n-2}} \begin{split}
 0 \to \wedge^2\sP_{n-k+1}  \to \wedge^2\sL'_{n-k+1} \to  \sL'_{n-k+1}\otimes \sP_{n-k}  \to S^2\sP_{n-k}\\
 0 \to S^2\sP_{n-k+2}  \to \wedge^2\sL'_{n-k+2} \to  \sL'_{n-k+2}\otimes \sP_{n-k+1}  \to \wedge^2\sP_{n-k+1}\\
 0 \to \wedge^2\sP_{n-k+3}  \to \wedge^2\sL'_{n-k+3} \to  \sL'_{n-k+3}\otimes \sP_{n-k+2}  \to S^2\sP_{n-k+2}\\
 \vdots  \hskip 3cm \vdots \hskip 3cm \vdots\\
 0 \to S^2\sL'_{n+1}  \to S^2\sL'_{n} \to  \sL'_{n}\otimes \sP_{n-1}  \to \wedge^2\sP_{n-1}\\
 \end{split}
 \end{equation}
 $H^{2k-3}_*(S^2\sP_{n-k}(M))$ will vanish provided $H^{2k-3}_*(\sP_{n-k})$, $H^{2k-1}_*(\sP_{n-k+1})$, $\dots,$ $H^{4k-5}_*(\sP_{n-1})$  and $H^{4k-3}_*(S^2\sL'_{n+1})$ are all zero. $H^{4k-3}_*(S^2\sL'_{n+1})=0$ since $n> 4k-3$. For the others, $H^{2k-3+2i}_*(\sP_{n-k+i}) =0$ since $n-k+i > 2k-3+2i$ when $0 \leq i \leq k-1$.
We have concluded that   $H^{2k-2}_*(\sD) =0$ from the second sequence, contradicting the earlier result of being non-zero.
 \vskip 5 pt
  \noindent {\bf Case 2}  The case where $N$ is non-cyclic, $k>2$ is even.

 This is very similar to Case 1. We use the same sequence (\ref{second}). Now $H^4_*(\wedge^2\sP_2(N)) \neq 0$ (see Theorem \ref{P7}, Case 2). Hence $H^{2k}_*(\wedge^2\sP_k(N))\neq 0$, since $n>2k$. $H^{2k}_*(\sP)$ and $H^{2k}_*(\wedge^2\sP)$ are zero, since $k < 2k < n-k$, hence $H^{2k-1}_*(\sD) \neq 0$.

 Again, $H^{2k-1}_*(\sP_{n-k}(M)\otimes \sF)=0$ since $2k-1 \neq n-k$ and $H^{2k-1}_*(\sP_{n-k}(M)\otimes \sP) =0$ since
  $n-k>3k-1$ and $k<2k-1$. Lastly, $H^{2k-2}_*(S^2\sP_{n-k}(M)) =0$ since $n>4k-2$ and $n-k+i > 2k-2+2i$ when $0 \leq i \leq k-1$. Hence $H^{2k-1}_*(\sD)$ is also equal to $0$.

 \vskip 5 pt
  \noindent {\bf Case 3}  The case where  $k$ is odd.

 Whether $N$ is cyclic or not, starting with $S^2\sL_0$, we get $H^2_*(\wedge^2\sP_1(N))\neq 0$. Since $k$ is odd, this results in $H^{2k}_*(\wedge^2\sP_k(N))\neq 0$ as in Case 2. We can now use sequence (\ref{second}) and copy the proof in Case 2.

  \end{proof}

\end{document}